\documentclass[12pt,reqno]{amsart}

\usepackage[latin1]{inputenc}
\usepackage{amsmath}
\usepackage{amsfonts}
\usepackage{amssymb}
\usepackage{graphics}

\usepackage{enumerate}
\usepackage{amssymb,amsmath,amsthm,amscd}
\usepackage{latexsym,verbatim,graphicx,amsfonts}

\usepackage{amscd}
\usepackage{amsmath}
\usepackage{amssymb}
\usepackage{amsthm}
\usepackage{latexsym}
\usepackage{verbatim}
\usepackage{hyperref}

\theoremstyle{plain}
\newtheorem{theorem}{Theorem}[section]
\newtheorem{thm}[theorem]{Theorem}

\newtheorem{lemma}[theorem]{Lemma}

\theoremstyle{definition}

\newtheorem{defn}[theorem]{Definition}
\newtheorem{ex}[theorem]{Example}
\theoremstyle{remark}
\newtheorem{rem}[theorem]{Remark}

\newcommand{\R}{\mathbb{R}}
\newcommand{\C}{\mathbb{C}}

\newcommand{\T}{\mathbb{T}}
\newcommand{\D}{\mathbb{D}}
\newcommand{\N}{\mathbb{N}}
\DeclareMathOperator*{\dist}{dist}

\title[Inner functions and optimal approximants]{Remarks on inner functions and optimal approximants}
\author[B\'en\'eteau]{Catherine B\'en\'eteau}
\address{Department of Mathematics, University of South Florida, 4202 E. Fowler Avenue, Tampa, FL 33620, USA.}
\email{cbenetea@usf.edu}
\author[Fleeman]{Matthew Fleeman}
\address{Department of Mathematics, Baylor University, One Bear Place \#97328, Waco, TX 76798-7328, USA.}
\email{Matthew$\underline{\,\,\,}$Fleeman@baylor.edu}
\author[Khavinson]{Dmitry Khavinson}
\address{Department of Mathematics, University of South Florida, 4202 E. Fowler Avenue, Tampa, FL 33620, USA.}
\email{dkhavins@usf.edu}

\author[Seco]{Daniel Seco}
\address{Instituto de Ciencias Matem\'aticas, Calle Nicol\'as Cabrera, UAM, 28049 Madrid,
Spain.} \email{dseco@mat.uab.cat}
\author[Sola]{Alan A. Sola}
\address{Department of Mathematics, Stockholm University, SE-106 91 Stockholm, Sweden.}
\email{sola@math.su.se}

\date{\today}
\subjclass[2010]{Primary 46E22; Secondary 30J05.}

\begin{document}

\maketitle

\begin{abstract}
We discuss the concept of inner function in reproducing
kernel Hilbert spaces with an orthogonal basis of monomials and examine connections between inner functions and optimal polynomial approximants to $1/f$, where $f$ is a function in the space. We revisit some classical examples from this perspective, and show how a construction of Shapiro and Shields can be modified to produce inner functions.
\end{abstract}

\section{Introduction}\label{Intro}

The notion of inner function is a central concept in operator-theoretic function theory, and has played a significant role in the description of the invariant
subspaces of the shift operator. The Hardy space $H^2$ of the disk
consists of all the functions $f$ analytic in the open unit disk $\D$ that satisfy the norm boundedness condition
$$\|f\|^2_{H^2}=\sup_{0< r < 1} \frac{1}{2\pi} \int_0^{2 \pi} |f(re^{i \theta})|^2 \, d \theta < \infty.$$
It is well-known (see, e.g., \cite{Duren}) that such functions have non-tangential boundary values on the unit circle $\T$ almost everywhere.
 A bounded analytic function $f \colon \D \to \C$ is called \emph{inner} if
its boundary values satisfy
\begin{equation}\label{HardyInner}
|f(e^{i \theta})|=1 \quad a.e. \,\, \theta \in [0, 2\pi).
\end{equation}
Beurling \cite{Beurling} showed that every closed $z$-invariant subspace is generated by an inner function.  Implicit in his proof is the fact that if $G$ is an inner function, then $G$ is orthogonal to $z^jG$ for every integer $j \geq 1.$  In his analysis of $z$-invariant subpaces in the Dirichlet space $D$, which is the space of all analytic functions in the disk whose derivative is square integrable with respect to area measure, Richter \cite{Richter} showed that as in the case of the Hardy space, the invariant subspaces are generated by a single function which satisfies the same orthogonality relationships.

The situation in the Bergman spaces turned out to be considerably more complicated.  The Bergman space $A^2$ is the set of all analytic functions in the disk whose modulus squared is integrable with respect to area measure.  In 1991, Hedenmalm \cite{Hedenmalm} noticed that if $M$ is a $z$-invariant subspace, then the function that is a solution to a particular extremal problem related to $M$ plays the same role as that of the $H^2$ inner functions.  Moreover, this function is a so-called \emph{contractive divisor}.  Duren, Khavinson, Shapiro and Sundberg \cite{DKSS} extended this result to all $A^p$ spaces of the disk.   In \cite{Koren2}, Korenblum coined the term ``$A^2$-inner" to describe Hedenmalm's extremal function $G$, and in \cite{ARS}, Aleman, Richter and Sundberg  proved an analogue of Beurling's theorem for the Bergman space, namely, that $z$-invariant subspaces $M$ of the Bergman space are generated by the ``wandering subspace" $M \ominus zM$. They defined an $A^2$-inner function as a function $G$ that has norm $1$ and is such that
$G$ is orthogonal to $z^jG$ for every integer $j \geq 1.$ More generally, such orthogonality conditions appear naturally in connection with the study of wandering subspaces in operator theory \cite{Halmos,KLS,LS,CDS}, and more concretely in many papers dealing with the 
shift operator acting on spaces of analytic functions. 

Another nice property that functions in the Hardy space satisfy is that any function factors as the
product of an inner and an outer function.  Outer functions $f \in H^2$ are defined by the condition that
$\log|f(0)| = \frac{1}{2\pi} \int_0^{2 \pi} \log|f(e^{i \theta})| \, d \theta$, and in the Hardy space, such functions are always \emph{cyclic} (and vice versa), that is, their polynomial multiples generate the whole space.  This factorization was an important tool in Beurling's characterization of the invariant subspaces of the Hardy space.  In \cite{Koren2}, Korenblum defined a notion of an $A^2$-outer function based on the concept of domination, and he proved that cyclic functions are outer in this sense, and the converse was proved in \cite{ARS}. However, cyclic functions in the Bergman space are still not well-understood, and it is an open problem to characterize cyclic functions in other spaces of analytic functions such as the Dirichlet space.

In \cite{BCLSS1}, the authors proposed to investigate cyclic functions in a large class of Hilbert spaces of analytic functions via the study of
\emph{optimal polynomial approximants}, that is, polynomials $p$ minimizing the norm
$\|pf-1\|$ among all polynomials in the space $\mathcal{P}_n$ of
polynomials of degree less or equal to $n$. This was done in the
generality described here in \cite{FMS1}. Optimal polynomial
approximants are given as the unique solution to a linear system of
the form $Mc=b$
where $M$ is a matrix whose elements are given as
\begin{equation}\label{LinearApprox2}
M_{j,k}=\left\langle z^j f, z^k f\right\rangle,
\end{equation}
$\left\langle\cdot,\cdot \right\rangle$ denotes the inner product in the space $H$, $c$ is
the vector of unknown coefficients of the optimal polynomial
approximant $p_n^*$; and $b$ is the vector given by
\begin{equation}\label{DefB}
\left(b_j\right)_{j=0}^n= \left(\left\langle 1,
z^jf\right \rangle \right)_{j=0}^n= \left(
\overline{f(0)},0,...,0\right).\end{equation}
Optimal approximants were further studied in a subsequent series of papers \cite{BKLSS,BKLSS2}, and it seems worthwhile to isolate additional properties of functions $G$ that satisfy the 
orthogonality relations 
\begin{equation}\label{Hinner0}
\left\langle z^j G, G \right\rangle =0, \,\,\,\, j=1, 2, \ldots
\end{equation}
from this perspective.

The goal of this paper is to discuss the notion of inner function in a wide class of Hilbert spaces, with a special focus on describing the optimal polynomial approximants associated with such functions.  We point out that certain functions, which we dub Shapiro-Shields functions after the authors in whose papers they first appeared, are inner and can be viewed as analogues of finite Blaschke products. Finally, we show how distances between $1$ and subspaces generated by inner functions can be computed using elementary linear algebra methods.

In Section 2, we give the relevant definitions and prove a characterization of inner functions in terms of optimal polynomial approximants: the optimal approximant of all orders of an inner function turn out to be equal to a single constant.  In Section 3, we use a slightly modified version of a construction of Shapiro and Shields to give examples of inner functions in Hilbert spaces of weighted Hardy type vanishing on prescribed finite sets.  We also compute distances between the function $1$ and invariant subspaces generated by an inner function. Finally, in Section 4, we discuss some examples of inner functions.

\noindent\textbf{Acknowledgements.}  The authors would like to thank
Constanze Liaw for many helpful conversations. B\'en\'eteau and
Khavinson are grateful to Stockholm University for supporting their
visit during work on this project. The work of Seco is supported by
Ministerio de Econom\'ia y Competitividad Project MTM2014-51824-P.

\section{Characterization of inner functions via optimal approximants}

From now on, given a sequence of real positive numbers
$\omega=\{\omega_k\}_{k \in \N}$, let $H=H^2_{\omega}$ be
the Hilbert space of holomorphic functions $f \colon \D\to \C$ with
Taylor coefficients $\{a_k\}_{k \in \N}$ endowed with norm
\begin{equation}\label{NormDef}\|f\|_{\omega} := \left(\sum_{k=0}^{\infty} |a_k|^2 \omega_k \right)^{1/2} <
\infty,\end{equation} and equipped with the inner product
$\left\langle\,,\right\rangle=\left\langle\,,\right\rangle_{\omega}$ induced by the norm. Without loss
of generality we will assume $\omega_0=1$. Furthermore, let us
restrict ourselves to the class of weights $\omega$ with
\begin{equation}\label{DefOmega} \lim_{k
\rightarrow \infty} \frac{\omega_k}{\omega_{k+1}} =1. \end{equation}
Condition \eqref{DefOmega} ensures that functions that are
holomorphic on a disk of radius strictly greater than 1 are elements
of $H$, that all elements of $H$ are holomorphic on $\D$, and that
both the forward and backward shifts are bounded operators. Note that the Hardy space, the Bergman space, and the Dirichlet space are all examples of spaces $H^2_{\omega}$ with appropriate choice of weights $\omega$. See \cite{Duren, Shields, HedKorZhu, DS, EFKMR} for treatments of these spaces.

Spaces defined as in \eqref{NormDef} are examples of
\emph{reproducing kernel Hilbert spaces} (from now
on, \emph{RKHS}) over the disk $\D$. This means
that for any point $z_0 \in \D$, point evaluation is a bounded
functional, and there exists a function $k_{z_0} \in
H^2_\omega$ such that for any function $g \in
H^2_\omega$ we have the reproducing property
\begin{equation}\label{RepProp}
g(z_0)=\langle g,k_{z_0}\rangle.
\end{equation}
Using the usual representation of a reproducing kernel in terms of an orthonormal basis, we find that
\begin{equation}\label{eqn001}
k_{z_0}(z)=\sum_{k=0}^\infty \frac{\overline{z_0^k} z^k}{\omega_k}.
\end{equation}
See \cite{Paul} for a primer on the theory of RKHS.

\begin{defn}\label{Hinner}
A function $f \in H$ is called \emph{$H-$inner} (or simply,
\emph{inner}) if $\|f\|_H=1$ and for all $j > 0$:
\begin{equation}\label{Hinner1}
\langle z^j f, f \rangle_H =0.
\end{equation}
\end{defn}
Note that the standard Hardy inner functions are inner in this sense, as are Bergman-inner functions.  We will say a closed subspace $M \subset H$ is $z$-invariant in $H$ if $zM \subset M$.  For a function $f \in H$, we will write $[f]$ for the subspace generated by $f$ under the unilateral shift, that is, the closure of all polynomial multiples of $f$ in the norm of $H$.

A related concept is that of cyclicity.  A function is called \emph{cyclic} (in $H$, for
the operator of multiplication by $z$) if $[f]$ is equal to the whole space $H$.
Because polynomials are dense in $H$, the function 1 is cyclic,
and hence a function $f$ is cyclic if and only if $1 \in [f]$. That
is, we can define $f$ to be cyclic
 if there exists a
sequence $\{p_n\}_{n\in \N}$ of polynomials such that
\begin{equation}\label{CycFunct}
\lim_{n \rightarrow \infty} \|p_n f-1\|^2_H = 0.
\end{equation}

The following theorem is known for many spaces and appears in different forms in several places, including in \cite{Richter, ARS}, but we include it for completeness and to emphasize that it holds thanks to general principles governing these Hilbert spaces, rather than to particular characteristics of the individual spaces.

\begin{thm}\label{Thm1}
Let $M$ be $z$-invariant in $H$, and assume that there exists a function in $M$ that does not vanish at $0$. Then there is a unique solution
$G= h/\sqrt{h(0)}$ to the extremal problem
\begin{equation}\label{extremal}
\sup \left\{ \mathrm{Re}\left( g(0) \right): g \in M, \, \|g\| \leq 1 \right\},
\end{equation}
 where $h$ is the
orthogonal projection of 1 onto $M$.  Moreover, $G$ is an $H-$inner function. Conversely, if $G$ is a (non-constant) $H-$inner function, then $G$ generates a proper $z$-invariant subspace and solves (up to multiplication by a unimodular constant) the extremal problem \eqref{extremal} for $M = [G]$.
\end{thm}

\begin{proof}

Let $M$ be $z$-invariant in $H$ and let $h$ be  the orthogonal
projection of 1 onto $M$. Then since $1 - h \perp h,$ we have that $
\|1-h\|^2 + \|h\|^2 = 1,$  $\langle 1 - h, 1-h \rangle = 1 - h(0),$
and $ 0 = \langle z^j h , 1 - h \rangle = - \langle  z^j h , h
\rangle, j=1,2,\ldots$  Therefore $\|h\|^2 = h(0),$ so letting $G(z)
= h(z) / \sqrt{h(0)}$, we see that $G$ has norm $1$ and is inner.

Now let $g \in M.$  Then $ 0 = \langle g, 1-h \rangle = g(0) - \langle g, h \rangle,$ or $\langle g, h \rangle = g(0),$ so in fact $h(z) = K_M(z, 0),$ the reproducing kernel of $M$ at $0.$  If in addition $\|g\| \leq 1,$ then by the Cauchy-Schwarz inequality, $\mathrm{Re}(g(0)) \leq \|h\| = \sqrt{h(0)}.$  Therefore, by the above, $G$ solves the extremal problem \eqref{extremal}.

Uniqueness follows by a standard argument in extremal problems in Hilbert spaces (see, e.g, \cite{DS}):  it is easy to see that the extremal problem   \eqref{extremal} is equivalent to the problem of minimizing the norm of all functions $g \in M$ such that $g(0) = 1$.  This last collection of functions forms a convex set in $H$, and convex sets in Hilbert spaces have unique elements of minimal norm.

Conversely, suppose $G$ is a (non-constant) inner function.  Suppose the invariant subspace $[G] = H.$  Then $ 1 \in [G],$ and therefore there exist a sequence of polynomials $p_n$ such that $p_n G \rightarrow 1$ in $H.$  Since $G$ is inner, $\langle G, p_n G \rangle = \overline{p_n(0)}.$ Taking limits of both sides as $n \rightarrow \infty$ forces $G(0) = 1/\overline{G(0)},$ or $|G(0)|^2 = 1.$ Since $\|G\|= 1,$ this can only happen if $G$ is a constant.  Thus non-constant inner functions generate proper invariant subspaces.

Moreoever, for any polynomial $p$, since $G$ is inner, $\langle (p - p(0))G, G \rangle = 0,$ or $\langle pG, G \rangle = p(0),$ that is, $\overline{G(0)} G(z)$ is the reproducing kernel at $0$ for the space $[G].$  Thus, by the above discussion, $G$ is a unimodular multiple of the extremal solution to problem \eqref{extremal}.
\end{proof}

Now let us examine how inner functions relate to optimal approximants.  The study of optimal approximants in the context of the Dirichlet-type
spaces $D_{\alpha}$
was initiated by the authors in \cite{BCLSS1}, who were interested in cyclic functions.
They proposed to examine cyclicity (in a
certain family of spaces) via the study of \emph{optimal polynomial
approximants}, that is, polynomials $p$ minimizing the norm
$\|pf-1\|$ among all polynomials in the space $\mathcal{P}_n$ of
polynomials of degree less or equal to $n$.  However, this definition
makes sense for \emph{all} functions in the space, not only the cyclic functions. Indeed, the theorem below indicates what these optimal approximants are for \emph{inner} functions.  If we denote by $\mathcal{P}_n$ the space of polynomials of degree less or
equal to $n$, it is clear that if $p_n^*$ is the optimal approximant of degree $n$, then $p^*_n f$ is the orthogonal
projection of 1 onto the space $\mathcal{P}_n f$.

\begin{thm}\label{main}   Let $f \in H$, not identically $0$, and for each $n$, let $p_n^*$ be the optimal approximant of degree $n$ of $1/f$.
\begin{itemize}
\item[(a)]  If $h$ is the orthogonal
projection of 1 onto $[f]$, then $h$ is the unique function such that $\|p^*_n f -h \| \rightarrow 0$ as $n \rightarrow \infty$.
\item[(b)] If $f$ is inner, then all the optimal approximants are constants: $p_n^*(z) = \overline{f(0)}.$
\item[(c)] If $f$ is inner, then $\|p_n^*f - 1 \|^2 = \dist^2_H (1, [f]) = 1 - |f(0)|^2.$
\end{itemize}
\end{thm}

\begin{proof}

To prove (a), notice that $\bigcup_{n \in \N}
\mathcal{P}_n f $ is dense in $[f]$. Since $H$ is a
Hilbert space, the orthogonal projection of 1 onto $\mathcal{P}_n
f$, $p^*_n f$, must converge to the orthogonal projection of 1 onto
$[f]$, that is, $h$. (This fact was already noticed by the authors in \cite{BKLSS2}.)

To prove (b),  notice that for an inner
function $f$, the matrix $M$ whose entries $M_{j,k}$ are given by \eqref{LinearApprox2} has zeros in all positions of the first
column and row (except for the first position, $j=k=0$, where we have
$M_{0,0}=1$). Therefore, the inverse $M^{-1}$ of $M$ also satisfies this property. This tells us that the optimal approximants $p^*_n$ (for all
$n \in \N$) are all the same constant, given by
\begin{equation}\label{PolysInner}
p_n^*(z) = p_0^*(z) = \overline{f(0)}.
\end{equation}
But then any $H$-inner functions $f$ and its corresponding optimal
polynomial approximants $p^*_n$ have  for all $n \in \N$ the
property that
\begin{equation}\label{eqn101}\|p_n^* f-1\|^2= 1-|f(0)|^2,\end{equation}
and by definition of $p_n^*$, this quantity equals $\dist^2_H (1, [f]),$ which proves (c).
\end{proof}

Note that functions in $H$ always admit a kind of a weak factorization, where the ``outer" factor can be expressed in terms of the optimal approximants.  More specifically, we have the following.

\begin{thm}\label{factorization}
 Let $f \in H$ with $f(0) \neq 0$ and for each $n$, let $p_n^*$ be the optimal approximants of degree $n$ of $1/f$.  Let $G$ be the solution to the extremal problem
\begin{equation}\label{ext2}
\sup \left\{ \mathrm{Re}\left( g(0) \right): g \in [f], \, \|g\| \leq 1 \right\}.
\end{equation}
Then there exists a function $F$, analytic in the disk, such that $f(z) = G(z)F(z)$, and
\[F(z)= \lim_{n \rightarrow \infty} \frac{\overline{f(0)}}{p_n^*(z)}.\]
\end{thm}

\begin{proof}  Suppose $f \in H$, and let $G$ be the solution to the extremal problem \eqref{ext2}.  Then $[G] \subset [f]$, and therefore all the zeros of $G$ must also be zeros of $f$, and so $F: = f / G$ is an analytic function in the disk. Letting $p_n^*$ be the optimal approximants of degree $n$ of $1/f$, by Theorem \ref{main}, $\|p^*_n f -h \| \rightarrow 0$ as $n \rightarrow \infty$, where $h$ is the orthogonal
projection of 1 onto $[f]$.  By theorem \ref{Thm1}, $G(z) = h(z) /\sqrt{h(0)}.$  Therefore, since norm convergence implies pointwise convergence,
$$ \lim_{n \rightarrow \infty} p_n^*(z) G(z) F(z) = \sqrt{h(0)} G(z) = \overline{f(0)} G(z),$$ and therefore the conclusion follows.
\end{proof}

\begin{rem}
It is natural to ask whether the function $F$ appearing in Theorem \ref{factorization} is truly ``outer", in the sense of being cyclic in $H$; in particular, we would require that $F\in H$. This, unfortunately, need not be the case in general: it is known, see \cite{HZ}, that factorization fails in some weighted Bergman spaces. While $G$ has precisely the same zeros as $f$ {\it inside} the disk, it may well happen that $G(\zeta)=0$ for some $\zeta \in \mathbb{T}$, leading to rapid growth in $F$ at this boundary point, thus preventing $F$ from belonging to $H$. See \cite{HZ} for further details.
\end{rem}

\section{Shapiro-Shields Functions and distances to invariant subspaces generated by inner functions}\label{DistInv}
We now describe a method for constructing inner functions in the setting of a general $H^2_{\omega}$-space. The functions in question can be viewed as
analogues of finite Blaschke products in the Hardy space or of contractive divisors associated with finite zero sets in the Bergman space.

Let $Z=\{z_i\}_{i=1}^n$ be a collection of $n$ distinct points in $\D \backslash \{0\}$ and let $k_{w}$ denote the reproducing kernel of $H=H^2_\omega$ at the point $w\in \mathbb{D}$. We define $K_Z$ to be the $n\times n$ matrix with elements given by
\begin{equation}\label{MatrixK}
K_{i,j}=\langle k_{z_j}, k_{z_i} \rangle,
\end{equation}
for $i,j  =1,...,n$. In what follows, we write $|M|$ for the determinant of a square matrix $M$.

In their study of the classical Dirichlet space and its zero sets,
Shapiro and Shields \cite{ShaShi} exhibited functions that have a
prescribed finite set of zeros in the disk and maximize a certain
functional, and a modification of their construction yields inner
functions. With a standard Hilbert space argument, their
construction can be extended to infinite zero sets (see \cite{ShaShi}, and see for instance
\cite{EFKMR} for further developments).

\begin{defn}\label{ShaShiFunct}
Let $Z=\{z_i\}_{i=1}^n\subset \mathbb{D}$, let $k_{z_i}$ denote the reproducing kernel at $z_i$ in $H^2_{\omega}$, and set
\begin{equation}\label{FunctionF} f_Z(z):=
\begin{vmatrix}
1 & 1 &  \cdots  & 1 \\
\left(k_{z_i}(z)\right)_{i=1}^n & & K_Z &
\end{vmatrix}.\end{equation}
We define the \emph{Shapiro-Shields function} in
$H^2_\omega$, for the set $Z$, as
\begin{equation}\label{FunctionG}
g_Z(z):=\frac{f_Z(z)}{\|f_Z\|_{\omega}}.
\end{equation}
\end{defn}
It is readily seen that $g_Z(z_i)=0$ for $i=1,\ldots, n$ since the
first and $(i+1)$th column in the determinant defining $f_Z$ are
then equal. In order to show that Shapiro-Shields functions $g_Z$
are indeed $H$-inner, we need two auxiliary results.

\begin{lemma}\label{NormofF}
$$\|f_Z\|^2=\overline{|K_Z|}f_Z(0).$$
\end{lemma}
\begin{proof}
Consider the matrices
$\left(B_t\right)_{t=0}^n$ where
\begin{equation}\label{MatrixBt}
(B_t)_{i,j}= \begin{cases} K_{i,j} & \text{if }i \neq t \\ 1 &
\text{otherwise.}
\end{cases}
\end{equation}
In particular, $B_0=K_Z$.

We first expand the determinant defining $f_Z$ in terms of the $B_t$'s:
\begin{equation}\label{eqn301}
f_Z(z) = |B_0| + \sum_{t=1}^n k_{z_t}(z) |B_t| (-1)^t.
\end{equation}

This expresses the norm $\|f_Z\|$ in terms of a useful linear combination:
$$\langle f_Z , f_Z \rangle = \overline{|B_0|} \langle f_Z , 1 \rangle
+ \sum_{t=1}^n \overline{|B_t|} (-1)^t \langle f_Z , k_{z_t}
\rangle.$$ Two observations finish the proof: first, by the
definition of the norm in $H$, $\langle f_Z , 1 \rangle = f_Z(0)$
and second, by the reproducing property of the kernels, $\langle f_Z
, k_{z_t} \rangle = f_Z(z_t)$ is the determinant of a matrix where
two columns ($j=0$ and $j=t$) are identical, and hence $f_Z(z_t)=0$.
\end{proof}

The previous Lemma tells us, in particular, that $g_Z$ is well
defined provided $|B_0|\neq 0$. The next Lemma shows that this is indeed the case.
\begin{lemma}\label{Invertible}
Let $Z =\{z_i\}_{i=1}^n \subset \D$ be a set of $n$
distinct points. Then $K_Z$ is invertible.
\end{lemma}

\begin{proof}
If $n=1$, $K_Z=k_{z_1}(z_1)=\|k_{z_1}\|^2>0$.

Suppose $n \geq 2$. From \eqref{MatrixK} we see that
$K_Z$ is a Gram matrix, and therefore its determinant is non-zero if and only if the kernels $\{k_{z_i}\}_{i=1}^n$ are
linearly independent.
Seeking to arrive at a contradiction, let us assume linear independence fails. Without loss of
generality, suppose
\begin{equation}\label{lindep}
k_{z_n}=\sum_{i=1}^{n-1} \lambda_i k_{z_i}.\end{equation}
Let $L$ be the Lagrange interpolating polynomial that is equal to $1$ at $z_n$ and vanishes at $z_i$ for $1 \leq i \leq n-1.$ Taking the inner product of $L$ with the left hand side of \eqref{lindep} then gives $1$, while the inner product of $L$ with the right hand side of \eqref{lindep} gives $0$, a contradiction. 
\end{proof}

We arrive at the main result of this Section.
\begin{thm}\label{Blaschke}
Let $Z =\{z_i\}_{i=1}^n \subset \D \backslash \{0\}$ be a set of $n$
distinct points. Then the function $g_Z$ is $H$-inner and given by
\begin{equation}\label{eqn308}g_Z(z)=\frac{f_Z(z)}{\sqrt{\overline{|K_Z|}f_Z(0)}}.\end{equation}
\end{thm}

\begin{proof}
By the definition of $g_Z$ and by Lemma \ref{NormofF}, we have $\|g_Z\|=1$. It remains to verify that
\begin{equation}\label{eqn307}\langle z^k f_Z , f_Z\rangle =0, \quad k=1,2,\ldots.
\end{equation}

From the definition of $f_Z$ in \eqref{FunctionF} we see that
\begin{equation}\label{eqn309}
\langle z^k f_Z , f_Z \rangle = \overline{|B_0|} \langle z^k f_Z , 1
\rangle + \sum_{t=1}^{n} \overline{|B_t|} (-1)^t \langle z^k f_Z ,
k_{z_t} \rangle.
\end{equation}
The reproducing property of the kernel implies that $$\langle z^k f_Z
, k_{z_t} \rangle = z_t^k f_Z(z_t) =0,$$ where the second identity
comes from evaluating $f_Z(z_t)$ as a determinant, with columns
$j=0$ and $j=t$ being the same. Finally, observe that $\langle
z^k f_Z , 1 \rangle =0$ for all $k \geq 1$ by the definition of the
norm in $H$.
\end{proof}

Property (c) in Theorem \ref{Thm1} relates inner functions to
distances to invariant subspaces. Let us now
compute the distance from the function 1 to the invariant
subspace generated by a polynomial with zero set $Z=\{z_1,\ldots, z_n\}\subset \D\setminus \{0\}$.
Denote by $K_Z^{-1}$ the inverse of the matrix $K_Z$ and let $v^*$ be the transpose conjugate
of a vector $v$.

\begin{thm}\label{DistThm}
Let $Z\subset \D \backslash \{0\}$ be finite, and let
$\dist_H(1,[g_Z])$ denote the distance from $1$ to $[g_Z]$ in $H$,
where $g_Z$ is the Shapiro-Shields function associated with $Z$ and
let $K_Z$ be defined as in \eqref{MatrixK}.

Then, letting $v=(1,\ldots,1)$, we have
\begin{equation}\label{DistThm1}
\mathrm{dist}^2_H(1,[g_Z])=vK_Z^{-1}v^*.
\end{equation}
\end{thm}

\begin{proof} By Lemma
\ref{Invertible}, the matrix $K_Z$ is invertible and hence the right-hand side of \eqref{DistThm1} is well defined.

As mentioned in Theorem \ref{main}, the distance from the
function 1 to the invariant subspace generated by an inner function
$f \in H$ is given by $\sqrt{1-|f(0)|^2}$.

With the notation from \eqref{FunctionG}, this means that we just
need to compute $1-|g_Z(0)|^2$. From the definition of $g_Z$ and
Lemma \ref{NormofF} it is clear that
\begin{equation}\label{eqn400}
|g_Z(0)|^2= \frac{|f_Z(0)|^2}{\|f_Z\|^2} =
\frac{\overline{f_Z(0)}}{\overline{|K_Z|}}.
\end{equation}
Since $|g_Z(0)| \in \R$, we can ignore the conjugation. From now on,
denote by $v$ an $n$-vector with coordinates $(1,...,1)$, and $v^*$
its transpose (as in the statement of the Theorem). Bearing in mind
that $k_{z}(0)=1$ for any $z \in \D$, we obtain
\begin{equation}\label{eqn401}
1-|g_Z(0)|^2= 1 - \frac{\begin{vmatrix} 1 & v \\ v^* & K_Z
\end{vmatrix}}{|K_Z|}.
\end{equation}

Notice that $|K_Z|=\begin{vmatrix} 1 & v \\ 0 & K_Z
\end{vmatrix}$. This, applied to both the denominator and the numerator of the right-hand side of
\eqref{eqn401}, shows that
\begin{equation}\label{eqn402}
1-|g_Z(0)|^2 = \frac{\begin{vmatrix} 1 & v \\ 0 & K_Z\end{vmatrix}
-\begin{vmatrix} 1 & v \\ v^* & K_Z\end{vmatrix}}{\begin{vmatrix} 1
& v \\ 0 & K_Z\end{vmatrix}},
\end{equation}
and by elementary linear algebra, this yields
\begin{equation}\label{eqn403}
1-|g_Z(0)|^2 = \frac{\begin{vmatrix} 0 & v \\ -v^* &
K_Z\end{vmatrix}}{\begin{vmatrix} 1 & v \\ 0 & K_Z\end{vmatrix}}.
\end{equation}

By Lemma
\ref{Invertible}, the system
\begin{equation}\label{eqn404}
\begin{pmatrix} 1 & v \\ 0 & K_Z\end{pmatrix} c = \begin{pmatrix} 0 \\ -v^*\end{pmatrix}
\end{equation}
has a unique solution $c:=(c_0, \ldots, c_n)^T$ and, by applying Cramer's rule, we see that the left-hand side of
\eqref{eqn403} is equal to $c_0$.

The first equation in \eqref{eqn404} tells us that
\begin{equation}\label{eqn405}c_0 = -\sum_{j=1}^n c_j = -v \cdot c, \end{equation}
while the rest can be expressed in the simple form $K_Z c = -v^*$.
Applying Lemma \ref{Invertible}, we obtain
\begin{equation}\label{eqn407} c= -K_Z^{-1} v^*. \end{equation}
Substituting the value of $c_0$ obtained in \eqref{eqn405} into
\eqref{eqn407} finishes the proof.
\end{proof}

\section{Examples}
\begin{ex}[Zero-based invariant subspaces in $H^2$]
The Shapiro-Shields functions associated with $Z=\{z_1\}$, a singleton, are straight-forward to compute. In the case of the Hardy space $H^2$, the reproducing kernel is the Szeg\H{o} kernel
\[k_{z_1}(z)=\frac{1}{1-\bar{z}_1z},\]
and hence
\[f_Z(z)=\frac{1}{1-|z_1|^2}-\frac{1}{1-\bar{z}_1z}=\frac{\bar{z}_1}{1-|z_1|^2}\frac{z_1-z}{1-\bar{z}_1z}.\]
After normalizing, we obtain
\[g_Z(z)=\frac{\bar{z}_1}{|z_1|}\frac{z_1-z}{1-\bar{z}_1z},\]
a classical Blaschke factor. The associated distance is
\[\mathrm{dist}_{H^2}^2(1,[g_{\{z_1\}}])=1-|z_1|^2.\]

Let us turn to zero sets containing two points, $Z=\{z_1,z_2\}$. After a somewhat lengthy computation that can be carried out by hand or using computer algebra, we obtain
\[f_Z(z)=\frac{1}{(1-|z_1|^2)(1-|z_2|^2)}\frac{|z_1-z_2|^2}{|1-\bar{z}_1z_2|^2}\left(\bar{z}_1\frac{z_1-z}{1-\bar{z}_1z}\right)
\left(\bar{z}_2\frac{z_2-z}{1-\bar{z}_2z}\right).\]
Note that, at this stage, it is not clear that the Shapiro-Shields function factors as a product of singleton Shapiro-Shields functions. After normalizing as in Theorem \ref{Blaschke}, however, we do arrive at
\[g_Z(z)=\left(\frac{\bar{z}_1}{|z_1|}\frac{z_1-z}{1-\bar{z}_1z}\right)\left(\frac{\bar{z}_2}{|z_2|}\frac{z_2-z}{1-\bar{z}_2z}\right),\]
a Blaschke product. It is immediate that
\[\mathrm{dist}_{H^2}(1,[g_{\{z_1,z_2\}}])=1-|z_1z_2|^2.\]
\end{ex}

\begin{ex}[Zero-based invariant subspaces in $A^2$]
We next turn to the Bergman space, whose reproducing kernel is
\[k_{z_1}(z)=\frac{1}{(1-\bar{z}_1z)^2}.\]
Computing the corresponding $2\times 2$-determinant in the definition of $f_Z$ for the Bergman space, we find that
\[f_Z(z)=\frac{\bar{z}_1}{(1-|z_1|^2)^2}\frac{z_1-z}{1-\bar{z}_1z}\left(\frac{2-\bar{z}_1z-|z_1|^2}{1-\bar{z}_1z}\right).\]
Normalization as in Theorem \ref{Blaschke} gives us
\[g_Z(z)=\frac{1}{\sqrt{2-|z_1|^2}}\frac{\bar{z}_1}{|z_1|}\frac{z_1-z}{1-\bar{z}_1z}\left(\frac{2-\bar{z}_1z-|z_1|^2}{1-\bar{z}_1z}\right).\]
This recovers the well-known single-point extremal function for the Bergman space, see \cite[p.56]{Hedenmalm}. In terms of distances,
\[\mathrm{dist}_{A^2}^2(1,[g_{\{z_1\}}])=1-|z_1|^2(2-|z_1|^2)=(1-|z_1|^2)^2.\]
It is a priori clear that the distance to a single-zero invariant
subspace is smaller in the Bergman space than in $H^2$, but the
above computations give us a quantitative comparison.

The invariant subspace generated by two simple zeros can also be handled. The function $f_Z$ can be expressed as a linear combination of kernels, cf. also \cite{Hedenmalm}.
Moreover,
\[K_Z=\frac{1}{(1-|z_1|^2)^2(1-|z_2|^2)^2}\frac{|z_1-z_2|^2}{|1-z_1\bar{z}_2|^4}(2-|z_1+z_2|^2+2|z_1z_2|^2).\]
and
\begin{multline*}
f_Z(0)=\frac{1}{(1-|z_1|^2)^2(1-|z_2|^2)^2}\big(1-(1-|z_1|^2)^2-(1-|z_2|^2)^2\\+(1-|z_1|^2)^2(1-|z_2|^2)^2\frac{2\mathrm{Re}(1-\bar{z_1}z_2)^2-1}{|1-z_1\bar{z}_2|^4}\big).
\end{multline*}
After simplifying, for instance by using computer algebra, we obtain the Bergman-inner function
\begin{multline*}
g_Z(z)=C_ZB_Z(z)\left(1+\frac{1-|z_1|^2}{1-\bar{z}_1z}+\frac{1-|z_2|^2}{1-\bar{z}_2z}\right.\\\left.+\frac{1-|z_1|^2}{1-\bar{z}_1z}\frac{1-|z_2|^2}{1-\bar{z}_2z}\frac{|1-z_1\bar{z_2}|^2-(1-|z_1|^2)(1-|z_2|^2)}{|1-z_1\bar{z_2}|^2+(1-|z_1|^2)(1-|z_2|^2)}\right),
\end{multline*}
where $B_{Z}=b_{z_1}b_{z_2}$ is a Blaschke product and $C_Z$ is the constant
\begin{multline*}
C_Z=\left(3-|z_1|^2-|z_2|^2\right.\\\left.+(1-|z_1|^2)(1-|z_2|^2)\frac{|1-z_1\bar{z_2}|^2-(1-|z_1|^2)(1-|z_2|^2)}{|1-z_1\bar{z_2}|^2+(1-|z_1|^2)(1-|z_2|^2)}\right)^{-1/2}
\end{multline*}
This is a special case of a more general result of Hansbo: in \cite[Corollary 2.10]{Hansbo}, he obtains a formula for the $A^p$-extremal function associated with two zeros having arbitrary multiplicity.  Note that, as is well-known (see, e.g., \cite{DS}), the rational function complementing the Blaschke product $B_Z$ in the above formula is (up to a constant) the reproducing kernel at $0$ in the weighted Bergman space with weight $|B_Z|^2$.

For $A^2$ then, the distance to a two-zero invariant subspace is
\begin{multline*}
\mathrm{dist}^2_{A^2}(1,[g_{\{z_1,z_2\}}])=1-|z_1z_2|^2\Big(3-|z_1|^2-|z_2|^2\\ +(1-|z_1|^2)(1-|z_2|^2) \frac{|1-\bar{z}_1z_2|^2-(1-|z_1|^2)(1-|z_2|^2)}{|1-\bar{z}_1z_2|^2+(1-|z_1|^2)(1-|z_2|^2)}\Big).
\end{multline*}
One immediately sees that not only the radial position but also the angle influences the distance function in the Bergman space: for two zeros on the same radius, the distance is maximized by placing the zeros at the same point, and minimized by placing them antipodally.
\end{ex}

\begin{ex}[Optimal approximants for a product function]
The following example should be compared with Theorem \ref{main}, part (a).

For $\lambda\in \mathbb{D}$, let
\[f(z)=(1-z)b_{\lambda}(z)=(1-z)\frac{\lambda-z}{1-\bar{\lambda}z}.\]
Since $f$ contains a Blaschke factor, $[f]$ is a proper closed invariant subspace of $H^2$, but $f$ is not inner in the sense of Definition \ref{Hinner}.

It is natural to ask what the optimal approximants $p_n^*$ to $1/f$ look like. By choosing coefficients in $p=\sum_{k=0}^nc_kz^k$ in a way that minimizes the norm
expression $\|pf-1\|_{H^2}$, we obtain
\[p_0^*=\frac{1}{2}\overline{\lambda},\quad p_1^*=\frac{2}{3}\overline{\lambda}\left(1+\frac{1}{2}z\right), \quad \textrm{and}\quad
p_2^*=\frac{3}{4}\overline{\lambda}\left(1+\frac{2}{3}z+\frac{1}{3}z^2\right).\]
We now recognize the $p_n^*$ for $n=0,1,2$ as $\overline{\lambda}$-multiples of the optimal approximants associated with the function $1-z$, as computed in \cite{BCLSS1}.

This is in fact the case for all $n$, as we will now prove. First, note that for any polynomial $p$,
\[\|pf-1\|_{H^2}=\|(1-z)p-b^{-1}_{\lambda}\|_{H^2}.\]
Next, we expand $b^{-1}_{\lambda}$ in a Laurent series: there is a constant term $\overline{\lambda}$, and the remaining powers of $z$ are all negative. By orthogonality then,
minimizing $\|pf-1\|_{H^2}$ over polynomials of degree $n$ is equivalent to minimizing $\|(1-z)p-\overline{\lambda}\|_{H^2}$
The equality
\[\|(1-z)p-\overline{\lambda}\|_{H^2}=|\lambda|\left\|\frac{p(1-z)}{\overline{\lambda}}-1\right\|_{H^2}\]
shows that the optimal approximants to $1/f$ are indeed given by $\overline{\lambda}$ times the optimal approximants to $1/(1-z)$ for all $n$.

A straight-forward modification of the above argument identifies the $H^2$-optimal approximants to $f=(1-z)B_{\Lambda}$, where $B_{\lambda}=\prod_{k=1}^{N}b_{\lambda_k}$ is a finite Blaschke product:
the $n$th-order optimal approximant is
\[p_n^*=\left(\prod_{k=1}^N\overline{\lambda}_{k}\right)\cdot q_n^*,\]
where $q_n^*$ is the optimal approximant of order $n$ to $1/(1-z)$.

Despite its simplicity, this example illustrates the fact that in
$H^2$, optimal approximants are essentially determined by the outer
part, since the linear system giving the coefficients of the
polynomials depends only on the outer part, and the inner part only
affects the end result by multiplying the independent term by a constant.
\end{ex}

\begin{ex}[Singular inner functions]
The construction of Shapiro and Shields does not not produce inner functions with singular factors. However, it is instructive to examine distances associated with such functions as well.

Let us focus on the atomic case. For any $\sigma>0$,
\[S_\sigma(z)=\exp\left(-\sigma\frac{1+z}{1-z}\right)\]
is an inner function for the Hardy space $H^2$, and we have
\[\mathrm{dist}_{H^2}^2(1,[S_{\sigma}])=1-|S_{\sigma}(0)|^2=1-e^{-2\sigma}.\]

The function $S_{\sigma}$ is not $A^2$-inner, but Duren, Khavinson, Shapiro, and Sundberg computed the Bergman extremal function for the subspace generated by $S_{\sigma}$ in
\cite{DKSS2}. Using a limiting argument, they obtained
\[G_{\sigma}(z)=\frac{1}{\sqrt{1+2\sigma}}\left(1+\frac{2\sigma}{1-z}\right)S_{\sigma}(z).\]
Hence
\[\mathrm{dist}^2_{A^2}(1,[S_{\sigma}])=1-(1+2\sigma)e^{-2\sigma}.\]
As with the case of a single zero, the Bergman distance is smaller than the Hardy distance, but this time merely by a correction of the coefficient multiplying $e^{-2\sigma}.$
\end{ex}
\begin{rem}
As can be seen from the preceding two examples, there are pairs of invariant subspaces, one zero-based and one associated with a singular inner function, that are equidistant to $1$.
\end{rem}


\begin{thebibliography}{99}

\bibitem{ARS}  \textsc{Aleman, A., Richter, S.}, and \textsc{Sundberg, C.}  Beurling's theorem for the Bergman space,
\emph{Acta Math.} {\bf 177} (1996), 275-310.

\bibitem{BCLSS1} \textsc{B\'en\'eteau, C., Condori, A. A., Liaw, C., Seco, D.}, and
\textsc{Sola, A. A.}, Cyclicity in Dirichlet-type spaces and
extremal polynomials, \emph{J. Anal. Math.} {\bf 126} (2015),
259-286.

\bibitem{BKLSS} \textsc{B\'en\'eteau, C., Khavinson, D., Liaw, C., Seco, D.}, and
\textsc{Sola, A. A.}, Orthogonal polynomials, reproducing kernels,
and zeros of optimal approximants, \emph{J. London Math. Soc.}
{\bf 94} (2016), 726-746.

\bibitem{BKLSS2} \textsc{B\'en\'eteau, C., Khavinson, D., Liaw, C., Seco, D.}, and
\textsc{Simanek, B.}, Zeros of optimal polynomial approximants:
Jacobi matrices and Jentzsch-type theorems, \emph{preprint},
\url{http://arxiv.org/abs/1606.08615}.

\bibitem{Beurling} \textsc{Beurling, A.}, On two problems concerning linear
transformations in Hilbert space, \emph{Acta Math.} {\bf 81} (1949),
239-255.

\bibitem{CDS} \textsc{Carswell, B., Duren, P.}, and \textsc{Stessin, M.}, Multiplication invariant subspaces of the Bergman space,
\emph{Indiana Univ. Math. J.} {\bf 51} (2002), no. 4, 931-961.

\bibitem{Duren} \textsc{Duren, P.}, \emph{Theory of $H^p$ Spaces}, Academic Press, New
York-London 1970; Second Edition, Dover Publications, Mineola, N.Y.,
2000.

\bibitem{DKSS} \textsc{Duren, P., Khavinson, D., Shapiro, H.}, and \textsc{Sundberg, C.}, Contractive zero-divisors in Bergman spaces,
\emph{Pacific J. Math.} {\bf 157} (1993), 37-56.

\bibitem{DKSS2} \textsc{Duren, P., Khavinson, D., Shapiro, H.}, and \textsc{Sundberg, C.}, Invariant subspaces in Bergman spaces and the biharmonic equation, \emph{Michigan Math. J.} {\bf 41} (1994), 247-259.

\bibitem{DS} \textsc{Duren, P.} and \textsc{Schuster, A.}, \emph{Bergman Spaces}, American
Mathematical Society, Providence, R.I., 2004.

\bibitem{EFKMR} \textsc{El-Fallah, O., Kellay, K., Mashreghi, J.}, and
\textsc{Ransford, T.}, \emph{A Primer on the Dirichlet Space},
Cambridge Tracts in Math. {\bf 203}, Cambridge University Press,
2014.

\bibitem{FMS1} \textsc{Fricain, E., Mashreghi, J.}, and
\textsc{Seco, D.}, Cyclicity in reproducing kernel Hilbert spaces of
analytic functions, \emph{Comput. Methods Funct. Theory} (2014) Issue 14,
665-680.

\bibitem{Halmos}\textsc{Halmos, P.R.}, Shifts on Hilbert spaces, \emph{J. Reine Angew. Math.} {\bf 208} (1961), 102-112.

\bibitem{Hansbo}\textsc{Hansbo, J.}, Reproducing kernels and contractive divisors in Bergman spaces, \emph{Zap. Nauchn. Sem. S.-Peterburg. Otdel. Mat. Inst. Steklov. (POMI)} {\bf 232} (1996), Issled. po Linein. Oper. i Teor. Funktsii. 24, 174-198.

\bibitem{Hedenmalm}  \textsc{Hedenmalm, H.}, A factorization theorem for square area-integrable analytic functions,
\emph{J. Reine Angew. Math.} {\bf 422} (1991), 45-68.

\bibitem{HedKorZhu} \textsc{Hedenmalm, H., Korenblum, B.}, and \textsc{Zhu, K.},
\emph{Theory of Bergman Spaces}, Graduate Texts in Mathematics,
Springer-Verlag, New York, 2000.

\bibitem{HZ} \textsc{Hedenmalm, H.} and \textsc{Zhu, K.}, On the failure of optimal factorization  for certain weighted
Bergman spaces, \emph{Complex Variables Theory Appl.} {\bf 19} (1992), 165-176.

\bibitem{KLS} \textsc{Khavinson, D., Lance, T.}, and \textsc{Stessin, M.}, Wandering property in the Hardy space,
\emph{Michigan Math. J.} {\bf 44} (1997), no. 3, 597-606.

\bibitem{Koren2} \textsc{Korenblum, B.},  Outer functions and cyclic elements in Bergman spaces,
\emph{J. Funct. Anal.} {\bf 115} (1993), 104-118.

\bibitem{LS} \textsc{Lance, T. }, and \textsc{Stessin, M.}, Multiplication invariant subspaces of Hardy spaces,
\emph{Canad. J. Math.} {\bf 49} (1997), no. 1, 100-118.

\bibitem{Paul}\textsc{Paulsen,V.} and  \textsc{Raghupathi, M.}, \emph{An Introduction to the Theory of Reproducing Kernel Hilbert Spaces}, Cambridge Studies in Advanced Mathematics {\bf 152}, Cambridge University Press, 2016.

\bibitem{Richter} \textsc{Richter, S.}, Invariant subspaces of the Dirichlet shift,
\emph{J. Reine Angew. Math.} {\bf 386} (1988), 205-220.

\bibitem{Seco} \textsc{Seco, D.}, Some problems on optimal approximants,
\emph{Recent progress on operator theory and approximation in spaces
of analytic functions, Contemp. Mathematics} {\bf 679} (2016)
193-205, AMS.

\bibitem{ShaShi} \textsc{Shapiro, H.}, and \textsc{Shields, A. L.}, On the zeros of functions with finite Dirichlet
integral and some related function spaces, \emph{Math. Z.} {\bf 80} (1962)
217-299.

\bibitem{Shields}\textsc{Shields, A.L.} Weighted shift operators and analytic function theory, \emph{Topics in operator theory (C. Pearcy, ed.)}, Math. Surveys {\bf 13} (1974),49-128 .


\end{thebibliography}
\end{document}